\documentclass[preprint,review,10pt]{amsart}
\usepackage{amsmath,amssymb,amsfonts,xspace}
\newtheorem{theorem}{Theorem}[section]

\theoremstyle{definition}
\newtheorem{definition}[theorem]{Definition}
\newtheorem{example}[theorem]{Example}

\newtheorem{corollary}[theorem]{Corollary}
\newtheorem{lem}[theorem]{Lemma}

\theoremstyle{remark}

\numberwithin{equation}{section}

\begin{document}

\title{     multiplication $(m,n)$-hypermodules  
 }

\author{Mahdi  Anbarloei}
\address{Department of Mathematics, Faculty of Sciences,
Imam Khomeini International University, Qazvin, Iran.
}

\email{m.anbarloei@sci.ikiu.ac.ir}





\begin{abstract}
  The concept of   multiplication $(m,n)$-hypermodules was introduced   by Ameri and  Norouzi  in \cite{sorc2}. Here we intend to investigate extensively the multiplication $(m,n)$-hypermodules. Let $(M,f,g)$ be a $(m,n)$-hypermodule (with canonical $(m,n)$-hypergroups) over a commutative Krasner $(m,n)$-hyperring $(R,h,k)$. A $(m, n)$-hypermodule $(M, f, g)$ over $(R, h, k)$ is called a multiplication $(m, n)$-hypermodule if for each subhypermodule $N$ of $M$, there exists a hyperideal $I$ of $R$ such that $N =g(I, 1^{(n-2)}, M)$. 
\end{abstract}
\maketitle
\section{Introduction}
Algebraic hyperstructures represent a natural extension of classical algebraic structures. Introduction of the structures was first done by Marty at the $8^{th}$ Congress of Scandinavian Mathematicians in 1934 \cite {1}. Later on,various authors studied it such as by Mittas \cite{2},\cite{3} and Corsini \cite{4},\cite{5}.
 In 1904, Kasner introduce the notion of n-ary algebras in a lecture in a annual meeting \cite{6}. It is worth noting that Dorente wrote the first paper on the theory of n-ary groups in 1928 \cite{7}.
The notion of Krasner hyperrings was introduced by Krasner for the
first time in \cite{8}. Also, we can see some properties on Krasner hyperrings in \cite{9} and \cite{10}. In \cite{11}, Davvaz and Vougiouklis defined the notion of $n$-ary hypergroups which is a generalization of hypergroups in the sense of Marty. The concept of $(m,n)$-hyperrings was introduced in \cite{12}. Davvaz et al. introduced Krasner $(m, n)$- hyperrings as a generalization of $(m, n)$-rings and studied some results in this context in \cite{13}.  The notions of   maximal hyperideal, $n$-ary prime hyperideal, $n$-ary primary hyperideal and the radical of a hyperideal in a Krasner $(m,n)$-hyperring were introduced in \cite{sorc1}. Also, Ostadhadi and Davvaz studied the isomorphism theorems of ring theory and Krasner hyperring theory which are derived in the context of Krasner $(m, n)$-hyperrings in \cite{15}.
Anvariyeh et al  defined the notion of $(m, n)$-
hypermodules over $(m, n)$-hyperrings in \cite{16}. After, the concepts of    free and   canonical $(m, n)$-hypermodules were definded and studied in \cite{17} and \cite{18}.
In \cite{sorc2}, Ameri and Norouzi introduced a class of the $(m,n)$-hypermodules, which they defined as multiplication $(m,n)$-hypermodules. The $(m, n)$-hypermodule $\linebreak$ $(M, f, g)$ over $(R, h, k)$ is called a multiplication $(m, n)$-
hypermodule if for each subhypermodule $N$ of $M$, there exists a hyperideal $I$ of $R$ such that $N =g(I, 1^{(n-2)}, M)$. 
 
 The paper is orgnized as follows. In section 2, we have given some basic definitions and  results of $n$-ary hyperstructures which we need to develop our paper. In section 3, we have proved some important results of the multiplications $(m,n)$-hypermodule.
  In section 4, we have studied  the primary subhypermodules
in  a multiplication $(m,n)$-hyperodule. 
  In section 5, we  have investigated
 intersections and sums of the multiplication $(m,n)$-hypermodules.
\section{Preliminaries}
In this section, we recall some basic terms and definitions  concerning $n$-ary hyperstructures which we need to develop our paper.\\
A mapping $f : H^n \longrightarrow P^*(H)$
 is called an $n$-ary hyperoperation, where $P^*(H)$ is the
set of all the non-empty subsets of $H$. An algebraic system $(H, f)$, where $f$ is an $n$-ary hyperoperation defined on $H$, is called an $n$-ary hypergroupoid. \\
We shall use the following abbreviated notation:\\
The sequence $x_i, x_{i+1},..., x_j$ 
will be denoted by $x^j_i$. For $j < i$, $x^j_i$ is the empty symbol. In this convention

$f(x_1,..., x_i, y_{i+1},..., y_j, z_{j+1},..., z_n)$\\
will be written as $f(x^i_1, y^j_{i+1}, z^n_{j+1})$. In the case when $y_{i+1} =... = y_j = y$ the last expression will be written in the form $f(x^i_1, y^{(j-i)}, z^n_{j+1})$. 
 
For non-empty subsets $A_1,..., A_n$ of $H$ we define

$f(A^n_1) = f(A_1,..., A_n) = \bigcup \{f(x^n_1) \ \vert \  x_i \in  A_i, i = 1,..., n \}$.\\
An n-ary hyperoperation $f$ is called associative if

$f(x^{i-1}_1, f(x_ i ^{n+i-1}), x^{2n-1}_{n+i}) = f(x^{j-1}_1, f(x_j^{n+j-1}), x_{n+j}^{2n-1}),$ \\
hold for every $1 \leq i < j \leq n$ and all $x_1, x_2,..., x_{2n-1} \in H$. An $n$-ary hypergroupoid with the
associative $n$-ary hyperoperation is called an $n$-ary semihypergroup. 

An $n$-ary hypergroupoid $(H, f)$ in which the equation $b \in f(a_1^{i-1}, x_i, a_{ i+1}^n)$ has a solution $x_i \in H$
for every $a_1^{i-1}, a_{ i+1}^n,b  \in H$ and $1 \leq i \leq n$, is called an $n$-ary quasihypergroup, when $(H, f)$ is an $n$-ary
semihypergroup, $(H, f)$ is called an $n$-ary hypergroup.  

An $n$-ary hypergroupoid $(H, f)$ is commutative if for all $ \sigma \in \mathbb{S}_n$, the group of all permutations of $\{1, 2, 3,..., n\}$, and for every $a_1^n \in H$ we have $f(a_1,..., a_n) = f(a_{\sigma(1)},..., a_{\sigma(n)})$.
If an $a_1^n \in H$ we denote $a_{\sigma(1)}^{\sigma(n)}$ as the $(a_{\sigma(1)},..., a_{\sigma(n)})$.  We assume throughout this paper that all Krasner $(m,n)$-hyperrings are commutative.

If $f$ is an $n$-ary hyperoperation and $t = l(n- 1) + 1$, then $t$-ary hyperoperation $f_{(l)}$ is given by

$f_{(l)}(x_1^{l(n-1)+1}) = f(f(..., f(f(x^n _1), x_{n+1}^{2n -1}),...), x_{(l-1)(n-1)+1}^{l(n-1)+1})$. 
\begin{definition}
(\cite{13}). Let $(H, f)$ be an $n$-ary hypergroup and $B$ be a non-empty subset of $H$. $B$ is called
an $n$-ary subhypergroup of $(H, f)$, if $f(x^n _1) \subseteq B$ for $x^n_ 1 \in B$, and the equation $b \in f(b^{i-1}_1, x_i, b^n _{i+1})$ has a solution $x_i \in B$ for every $b^{i-1}_1, b^n _{i+1}, b \in B$ and $1 \leq i  \leq n$.\\
An element $e \in H$ is called a scalar neutral element if $x = f(e^{(i-1)}, x, e^{(n-i)})$, for every $1 \leq i \leq n$ and
for every $x \in H$. 

An element $0$ of an $n$-ary semihypergroup $(H, g)$ is called a zero element if for every $x^n_2 \in H$ we have
$g(0, x^n _2) = g(x_2, 0, x^n_ 3) = ... = g(x^n _2, 0) = 0$.
If $0$ and $0^ \prime $are two zero elements, then $0 = g(0^ \prime , 0^{(n-1)}) = 0 ^ \prime$  and so the zero element is unique. 
\end{definition}
\begin{definition}
(\cite{l1}). Let $(H, f)$ be a  $n$-ary hypergroup. $(H, f)$ is called a canonical $n$-ary
hypergroup if\\
(1) there exists a unique $e \in H$, such that for every $x \in H, f(x, e^{(n-1)}) = x$;\\
(2) for all $x \in H$ there exists a unique $x^{-1} \in H$, such that $e \in f(x, x^{-1}, e^{(n-2)})$;\\
(3) if $x \in f(x^n _1)$, then for all $i$, we have $x_i \in  f(x, x^{-1},..., x^{-1}_{ i-1}, x^{-1}_ {i+1},..., x^{-1}_ n)$.

We say that e is the scalar identity of $(H, f)$ and $x^{-1}$ is the inverse of $x$. Notice that the inverse of $e$ is $e$.
\end{definition}
\begin{definition}
(\cite{13})A Krasner $(m, n)$-hyperring is an algebraic hyperstructure $(R, f, g)$ which
satisfies the following axioms:\\
(1) $(R, f$) is a canonical $m$-ary hypergroup;\\
(2) $(R, g)$ is a $n$-ary semigroup;\\
(3) the $n$-ary operation $g$ is distributive with respect to the $m$-ary hyperoperation $f$ , i.e., for every $a^{i-1}_1 , a^n_{ i+1}, x^m_ 1 \in R$, and $1 \leq i \leq n$,

$g(a^{i-1}_1, f(x^m _1 ), a^n _{i+1}) = f(g(a^{i-1}_1, x_1, a^n_{ i+1}),..., g(a^{i-1}_1, x_m, a^n_{ i+1}))$;\\
(4) $0$ is a zero element (absorbing element) of the $n$-ary operation $g$, i.e., for every $x^n_ 2 \in R$ we have 

$g(0, x^n _2) = g(x_2, 0, x^n _3) = ... = g(x^n_ 2, 0) = 0$.
\end{definition}
A non-empty subset $S$ of $R$ is called a subhyperring of $R$ if $(S, f, g)$ is a Krasner $(m, n)$-hyperring. Let
$I$ be a non-empty subset of $R$, we say that $I$ is a hyperideal of $(R, f, g)$ if $(I, f)$ is an $m$-ary subhypergroup
of $(R, f)$ and $g(x^{i-1}_1, I, x_{i+1}^n) \subseteq I$, for every $x^n _1 \in  R$ and  $1 \leq i \leq n$.
\begin{definition} (\cite{16})
Let $M$ be a nonempty set. Then $(M, f, g)$ is an $(m, n)$-hypermodule over an $(m, n)$-
hyperring $(R, h, k)$, if $(M, f)$ is an $m$-ary hypergroup and the map 

$g:\underbrace{R \times ... \times R}_{n-1} \times M\longrightarrow 
P^*(M)$\\
statisfied the following conditions:

$(i)\  g(r_1^{n-1},f(x_1^m))=f(g(r_1^{n-1}
,x_1),...,g(r_1^{n-1}
,x_m))$

$(ii)\  g(r_1^{i-1},h(s_1^m),r_{i+1}^{n-1},x)=f(g(r_1^{i-1}
,s_1,r_{i+1}^{n-1},x),...,g(r_1^{i-1}
s_m,r_{i+1}^{n-1},x))$

$(ii)\  g(r_1^{i-1},k(r_i^{i+n-1}),r_{i+m}^{n+m-2},x)=
g(r_1^{n-1},g(r_m^{n+m-2},x))$

$ (iv) \ \{0\}=g(r_1^{i-1},0,r_{i+1}^{n-1},x)$.
\end{definition} 
If $g$ is an $n$-ary hyperoperation, $T_1,...,T_{n-1}$ are subsets of $R$ and $M^\prime \subseteq M$, we set

$g(T_1^{n-1},M^\prime)=\bigcup \{g(r_1^{n-1},m)\ \vert \ r_i \in T_i, 1 \leq i \leq n-1, m \in M^\prime \}.$\\
A canonical $(m,n)$-hypermodule $(M,f,g)$ is an $(m,n)$-ary hypermodule with a canonical m-ary hypergroup $(M,f)$ over a Krasner $(m,n)$-hyperring $(R,h,k)$.\\

Let $(M, f, g)$ be an $(m, n)$-hypermodule over $(R, h, k)$. A non-empty subset $N$ of $M$ is said to be  an $(m,n)$-subhypermodule of $M$ if $(N,f)$ is a $m$-ary subhypergroup of $(M,f)$ and $g(R^{(n-1)},N) \in P^*(N)$.

Let $(M, f, g)$ be an $(m, n)$-hypermodule and $N$ be a subhypermodule of $M$. Then the hyperideal $S_N$ is considered as follows:

$S_N=\{r \in R \ \vert \ g(r,1^{(n-2)},M) \subseteq N\}$
\begin{definition} (\cite{sorc3})
Let $M$ be an $(m, n)$-hypermodule over  $(R, h, k)$. A proper subhypermodule $K$ of $M$ is said to be
maximal, if for $N \leq M$ with $K \subseteq  N \subseteq  M$, we have either $K = N$ or $N = M$. 
\end{definition} 

\begin{definition} (\cite{sorc3})
Let $M$ be an $(m, n)$-hypermodule over  $(R, h, k)$. A proper subhypermodule $N$ of $M$ is said to be
$n$-ary primary, if $g(r_1^{n-1},x)\subseteq Q$ with $r_1^{n-1} \in R$ and $x \in M \backslash Q$, implies that

$g(k(r_1^{(t)},1^{(n-t)}),...,k(r_{n-1}^{(t)},1^{(n-t)}),M) \subseteq N$\\
for $t \leq n$, and for $t>n$ such that $t=l(n-1)+1$

$g(k_{(l)}(r_1^{(t)}),...,k_{(l)}(r_{n-1}^{(t)}),M) \subseteq N$\\
for some $t \in \mathbb{N}$ and $l >0$.  

\end{definition} 
\begin{definition}(\cite{sorc3})
Let $N$ be a subhypermodule of an  $(m, n)$-hypermodule $(M, f, g)$ over  $(R, h, k)$. Then the set

$R/N = \{f(x^{i-1}_1, N, x^m_{i+1}) \ \vert \  x^{i-1}_1,x^m_{i+1} \in M \}$\\
endowed with m-ary hyperoperation $f$ which for all $x_{11}^{1m},...,x_{m1}^{mm} \in R$

$F(f(x_{11}^{1 (i-1)}, N, x^{1m}_ {1(i+1)}),..., f(x_{m1}^{ m(i-1)}, N, x^{mm}_ {m(i+1)}))$ 

$\hspace{0.3cm}= \{ (f(t_1^{i-1}, N,t_{i+1}^m ) \ \vert \ t_1 \in f(x_{11}^{m1}),...,t_m \in  f(x_{1m}^{mm})\}$\\
and with $n$-ary hyperoperation $G:\underbrace{R \times ... \times R}_{n-1} \times M/N\longrightarrow 
 P^*(M/N)$ which for all $x_1^{i-1},x_{i+1}^m \in M$ and $r_1^{n-1} \in R$

$G(r_1^{n-1},f(x_1^{i-1},N,x_{i+1}^m))$

$\hspace{0.3cm}=\{f(z_1^{i-1},N,z_{i+1}^m\ \vert \ z_1 \in g(r_1^{n-1},x_1),..., z_m \in g(r_1^{n-1},x_m)\}$\\
is an  $(m, n)$-hypermodule over  $(R, h, k)$, and $(R/N, F, G)$  is called the quotient $(m, n)$-hypermodule  of $M$ by $N$. 
\end{definition} 

\begin{definition} (\cite{sorc4})
Let $(M, f, g)$ be an $(m, n)$-hypermodule over  $(R, h, k)$ and $P$ be a maximal hyperideal of $R$. Then

 (1) we define
\[\mathcal{X}_P(M) = \{m \in M \ \vert \  \exists p \in P; 0 \in g(h(1, -p, 0^{(m-2)}), 1^{( n-2)}, m)\}\]
 \[\mathcal{X}^P(M) = \{m \in M \ \vert \  \exists p \in P; \{0\} = g(h(1, -p, 0^{(m-2)}), 1^{( n-2)}, m)\}.\]
 In this paper we let $\mathcal{X}_P(M) = \mathcal{X}^P(M) $ for every maximal hyperideal $P$ of $R$.

(2) $M$ is called an n-ary $P$-cyclic if there exist $q \in  P$ and $m \in M$ such that $g(h(1, -q, 0^{(m-2)}),
1^{(n-2)}, M) \subseteq g(R, 1^{( n-2)}, m)$.

  Also, we let  $x \in  g(I, 1^{(n-2)}, m)$ implies that there exists $a \in I$ such that $\{x\} = g(a, 1^{(n-2)}, m)$,
and $\{0\} = h(r,-r, 0^{(m-2)})$, for all $x, m \in  M$, $r\in R$ and $I \subseteq R$. 
\end{definition}
\begin{definition} (\cite{sorc2})
For every nonzero element $m$ of $(m, n$)-hypermodule $(M, f , g)$ over $(R, h, k)$, we define

$F_m
= \{r \in  R  \ \vert \ 0 \in  g(r, 1^ {(n-2)}, m);  r \neq  0\}.$\\
It is clear that $F_m$ is a hyperideal of $(R, h, k)$.  The $(m, n)$-hypermodule $(M, f , g)$ is said to be  faithful,
if $F_m = \{0\}$ for all nonzero elements $m \in  M$, that is
$0 \in  g(r, 1^{( n-2)}, m)$ implies that  $r = 0$, for $r \in  R$. 
\end{definition}
\section{the multiplication $(m,n)$-hypermodules }
 The notion of multiplication  $(m,n)$-hypermodules was defined in \cite{sorc2}.  The $(m, n)$-hypermodule $(M, f, g)$ over $(R, h, k)$ is called a multiplication $(m, n)$-
hypermodule if for each subhypermodule $N$ of $M$, there exists a hyperideal $I$ of $R$ such that $N =g(I, 1^{(n-2)}, M)$. 
 \begin{example}
 $(m, n)$-hypermodule $(\mathbb{Z}, f, g)$ over $(m, n)$-hyperring $(\mathbb{Z}, h, k)$ is multiplication. (see Example 3.5 in \cite{sorc2})
  \end{example}
 
\begin{lem} \label{11}
Let $(M,f,g)$ be a multiplication  $(m,n)$-hypermodule over $(R,h,k)$ and let $J_{(m,n)}(R)$ be intersection of all maximal hyperideals of $R$. If   $\linebreak$ $M=g(I,1^{(n-2)},M)$  for some hyperideal $I \subseteq J_{(m,n)}(R)$, then $M=0$.
\end{lem}
\begin{proof}
If $m \in M$ then $g(R,1^{(n-2)},m)=g(A,1^{(n-2)},M)$ for some hyperideal $A$ of $R$. Then we get

$\hspace{1cm}
m \in g(1^{(n-1)},m) \subseteq g(R,1^{(n-2)},m)$

$\hspace{3.7cm}=g(A,1^{(n-2)},M)$

$\hspace{3.7cm}=g(A,1^{(n-2)},g(I,1^{(n-2)},M))$

$\hspace{3.7cm}=g(k(A,I,1^{(n-2)}),1^{(n-2)},M)$

$\hspace{3.7cm}=g(I,1^{(n-2)},g(A,1^{(n-2)},M))$

$\hspace{3.7cm}=g(I,1^{(n-2)},g(R,1^{(n-2)},m)$

$\hspace{3.7cm}=g(k(I,R,1^{(n-2)}),1^{(n-2)},m)$

$\hspace{3.7cm}=g(I,1^{(n-2)},m)$.\\
It imlies that there exists some $r \in I$ such that $m \in f(g(r,1^{(n-2)},m),0^{(m-1)})$. Then we get\\
$0 \in f(-g(r,1^{(n-2)},m),m,0^{(m-2)}) \subseteq f(g(-r,1^{(n-2)},m),g(1^{(n-1)},m),g(0,1^{(m-2)},m))$

$\hspace{4.9cm}=g(h(-r,1,0^{(m-2)}),1^{(n-2)},m)$.\\
By Corollary 4.15 in \cite{sorc1}, Since $r \in J_{(m,n)}(R)$ then $h(-r,1,0^{(m-2)})$ is invertible which means  there exists  $s \in R$ such that $k(s,h(-r,1,0^{(m-2)}),1^{(n-2)})=1$. Thus 

$\hspace{1.2cm}\{ 0\}=g(s,1^{(n-2)},0) \subseteq g(s,1^{(n-2)},g(h(-r,1,0^{(m-2)}),1^{(n-2)},m))$

$\hspace{4cm}=g(k(s,h(-r,1,0^{(m-2)}),1^{(n-2)}),1^{(n-2)},m)$

$\hspace{4cm}=g(1,1^{(n-2)},m)$

$\hspace{4cm}=g(1^{(n-1)},m)$

$\hspace{4cm}=\{m\}$.\\
Hence $m=0$ which implies $M=0$.
\end{proof}
\begin{theorem} \label{12}
Let $(M,f,g)$ be a nonzero multiplication $(m,n)$-hypermodule over $(R,h,k)$ such that  $R$ has  the only one maximal hyperideal. Then $M$ is cyclic. 
\end{theorem}
\begin{proof}
Let $(M,f,g)$ be a nonzero multiplication $(m,n)$-hypermodule over $(R,h,k)$. Assume that   $I$ is  the only  maximal hyperideal of $R$. The Lemma \ref{11} shows that $g(I,1^{(n-2)},M) \neq M$. Then there exists some $m \in M$ such that $m \notin g(I,1^{(n-2)},M)$. Thus we have $g(R,1^{(n-2)},m)=g(J,1^{(n-2)},M)$ for some hyperideal $J$ of $R$. Let $J \neq R$. It follows that  $J \subseteq I$. Hence we get $g(R,1^{(n-2)},m)=g(J,1^{(n-2)},M) \subseteq g(I,1^{(n-2)},M)$. Since $m \in g(1^{(n-1)},m) \subseteq g(R,1^{(n-2)},m)$, then we conclude that $m \in  g(I,1^{(n-2)},M)$ which is contradiction. Thus $J=R$ which implies $g(R,1^{(n-2)},m)=g(J,1^{(n-2)},M)=g(R,1^{(n-2)},M)=M$ and the proof is completed.
\end{proof}
\begin{theorem} \label{13}
If $(M,f,g)$ is a  multiplication hypermodule over $(R,h,k)$ then 
 $M/N_1$ with $N_1=g(J_{(m,n)}(R),1^{(n-2)},M)$ is a multiplication hypermodule over $(R,h,k)$ such that $J_{(m,n)}(R)$ is intersection of all maximal hyperideals of $R$. 
\end{theorem}
\begin{proof}
Let  $N_1=g(J_{(m,n)}(R),1^{(n-2)},M)$ in which $J_{(m,n)}(R)$ is intersection of all maximal hyperideals of $R$. Let $N$ be a subhypermodule of $M$ such that $N_1 \subseteq N$ and $N/N_1$ be a subhypermodule of $M/N_1$. Since $M$ is a multiplication hypermodule over $R$ then there exists some hyperideal $I$ of $R$ such that $N=g(I,1^{(n-2)},M)$. Since \\
$N/N_1=\{f(a_1^{m-1},N_1) \  \ \ \vert a_1^{m-1} \in N\}$

$\hspace*{0.5cm}=\{f(g(r_1,1^{(n-2)},x_1),...,g(r_{m-1},1^{(n-2)},x_{m-1}),N_1)  \vert r_1^{m-1} \in I, x_1^{m-1} \in N\}$

$\hspace*{0.5cm}=\{f(g(r_1,1^{(n-2)},x_1),...,g(r_{m-1},1^{(n-2)},x_{m-1}),f(N_1^{(m-1)},0 ) \vert r_1^{m-1} \in I, $

$\hspace*{0.8cm}x_1^{m-1} \in N\}$

$\hspace*{0.5cm}=\{f(f(g(r_1,1^{(n-2)},x_1),N_1,0^{(m-2)}),...,f(g(r_{m-1},1^{(n-2)},x_{m-1}),N_1 ,0^{(m-2)})$

$\hspace*{0.8cm},0) \vert r_1^{m-1}  \in I, x_1^{m-1} \in N\}$

$\hspace*{0.5cm}=\{f(\{f(t_1,N_1,0^{(m-2)})\ \vert \ t_1 \in g(r_1,1^{(n-2)},x_1)\},...,\{f(t_{m-1},N_1 ,0^{(m-2)})\ \vert \ t_{m-1} $

$\hspace*{0.8cm}\in g(r_{m-1},1^{(n-2)},x_{m-1})\},0) \vert r_1^{m-1}  \in I, x_1^{m-1} \in N\}$

$\hspace*{0.5cm}=F(\{f(t_1,N_1,0^{(m-2)})\ \vert \ t_1 \in g(r_1,1^{(n-2)},x_1)\},...,\{f(t_{m-1},N_1 ,0^{(m-2)})\ \vert \ t_{m-1} $

$\hspace*{0.8cm}\in g(r_{m-1},1^{(n-2)},x_{m-1})\},0)) $

$\hspace*{0.5cm}=F(G(r_1,1^{(n-2)},f(x_1,N_1,0^{(m-2)})) ,...,G(r_{m-1},1^{(n-2)},f(x_{m-1},N_1,0^{(m-2)})),0) $

$\hspace*{0.5cm}=G(r_1^{m-1},F(f(x_1,N_1,0^{(m-2)}),...,f(x_{m-1},N_1,0^{(m-2)}),0))$

$\hspace*{0.5cm}=\bigcup \{G(r,1^{(n-2)},f(x,N_1,0^{(m-2)})) \ \vert \ r=k(r_1^{m-1},1^{n-m+1}), r_1^{m-1} \in I, x \in$

$\hspace*{0.8cm}f(x_1^{m-1},0),x_1^{m-1} \in M\}$

$\hspace*{0.5cm}=G(I,1^{(n-2)},M/N_1)$.\\
Hence $M/N_1$ with $N_1=g(J_{(m,n)}(R),1^{(n-2)},M)$ is a multiplication hypermodule over $(R,h,k)$.
\end{proof}

\begin{theorem} \label{15}
Let $(M,f,g)$ be an $(m,n)$-hypermodule over $(R,h,k)$. $(M,f,g)$ is multiplication if and only if 

1) $\bigcap_{\lambda \in \Lambda} g(I_{\lambda},1^{(n-2)},M)=g(\bigcap_{\lambda \in \Lambda}[f(I_\lambda,S_0,0^{(m-2)}],1^{(n-2)},M)$ for all nonempty family of hyperideals $I_{\lambda}$ of $R$.

2) For every hyperideal $A$ of $R$ and the subhypermodule $N$ of $M$ with $N \subset g(A,1^{(n-2)},M)$, there exists a hyperideal $B \subset A$ such that $N \subseteq g(B,1^{(n-2)},M)$.
\end{theorem}
\begin{proof}
Since $M$ over $R/S_0$ is faithful, then we are done, by Theorem 3.11 in \cite{sorc2}. 
\end{proof}
\begin{definition}
An $(m,n)$-hypermodule $(M,f,g)$ over $(R,h,k)$ is called cofinitely generated if for each nonempty family of subhypermodules $N_{i}$ of $M$, $\bigcap_{i \in I} N_{i}=0$ then there exists a finite subset $\Lambda^\prime of \Lambda$ such that $\bigcap_{i \in I^\prime} N_i=0$.
\end{definition}
\begin{theorem} \label{16}
Let $(M,f,g)$ be a faithful multiplication $(m,n)$-hypermodule over $(R,h,k)$. Then $M$ is cofinitely generated if and only if $R$ is cofinitely generated.
\end{theorem}
\begin{proof}
$\Longrightarrow$ Let $(M,f,g)$ be a faithful multiplication $(m,n)$-hypermodule over $(R,h,k)$ and $\bigcap_{i \in I} J_i=0$ for a nonempty family of hyperideals $J_i$ of $R$. Then we have

 $\bigcap_{i \in I} g(J_i,1^{(n-2)},M)=g(\bigcap_{i \in I}J_i,1^{(n-2)},M)=g(0,1^{(n-2)},M)=0$\\by Theorem 3.11 in \cite{sorc2}. Since $M$ is cofinitely generated, then there exists a finite subset $I ^\prime$ of $I$ with

 $\bigcap_{i \in I^\prime}g(J_i,1^{(n-2)},M)=g(\bigcap_{i \in I^\prime}J_i,1^{(n-2)},M)=0$.\\
 This implies that $\bigcap_{i \in I^\prime}J_i=0$, since  $M$ is faithful. Then $(R,h,k)$ is cofinitely generated.\\
 $\Longleftarrow$ Let $(R,h,k)$ be cofinitely generated and $\bigcap_{i \in I} T_i=0$ for a nonempty family of subhypermodules $T_i$ of $M$. Since $M$ is a multiplication  $(m,n)$-hypermodule, then for each $i \in I$ there exists a hyperideal $J_i$ of $R$ such that $T_i=g(J_i,1^{(n-2)},M)$. Now we have 
 
 $g(\bigcap_{i \in I}J_i,1^{(n-2)},M) \subseteq \bigcap_{i \in I}g(J_i,1^{(n-2)},M)=\bigcap_{i \in I}T_i=0$.\\
 Since $M$ is faithful then $\bigcap_{i \in I}J_i=0$. Since $R$ is cofinitely generated, then there exists a finite subset $I^\prime$ of $I$ such that 
 $\bigcap_{i \in I^\prime}J_i=0$. Hence
 
 $\bigcap_{i \in I^\prime}T_i=\bigcap_{i \in I^\prime}g(J_i,1^{(n-2)},M)$
 
 $\hspace{1.3cm}=g(\bigcap_{i \in I^\prime}J_i,1^{(n-2)},M)$
 
 $\hspace{1.3cm}=g(0,1^{(n-2)},M)$
  
 $\hspace{1.3cm}=0$. \\Thus
 $(m,n)$-hypermodule $(M,f,g)$ is cofinitely generated. 
\end{proof}
Let $(M,f,g)$ be an $(m,n)$-hypermodule over $(R,h,k)$. We denote the family of all hyperideals $A$ of $R$ such that $g(A,1^{(n-2)},M)=M$ by $\omega$ and denote  intersection  of all members of $\omega$ by $\omega(M)$.
\begin{theorem} \label{17}
Let $(M,f,g)$ be an faithful multiplication $(m,n)$-hypermodule over $(R,h,k)$ and $\omega(M)=B$. Then:

(1) $x \in g(B,1^{(n-2)},m)$, for each $x \in M$. 

(2) $B=B.B$.

(3) $B \subseteq P$ or $h(B,P,0^{(n-2)})=R$, for each $n$-ary prime hyperideal $P$ of $R$.

(4) $M$ is a multiplication $(m,n)$-hypermodule over $B$. 

(5) $g(C,1^{(n-2)},M) \neq M$, for each proper hyperideal $C$ of $R$
\end{theorem}
\begin{proof}
(1) Let $B=\omega(M)=\bigcap_{A \in \omega} A$. Theorem 3.11 in \cite{sorc2} shows that $\linebreak$
$g(B,1^{(n-2)},M)=g(\bigcap_{A \in \omega} A,1^{(n-2)},M)=\bigcap_{A \in \omega}(g(A,1^{(n-2)},M)=\bigcap_{A \in \omega}M=M$. Let $x\in R$. Since $M$ is multiplication, then $g(R,1^{(n-2)},x)=g(D,1^{(n-2)},M)$ for some hyperideal $D$ of $R$. Then we get 

$x \in g(1^{(n-1)},x) \subseteq g(R,1^{(n-2)},x)$

$\hspace{2.4cm}=g(D,1^{(n-2)},M)$

$\hspace{2.4cm}=g(D,1^{(n-2)},g(B,1^{(n-2)},M))$

$\hspace{2.4cm}=g(k(D,1^{(n-2)},B),1^{(n-2)},M)$

$\hspace{2.4cm}=g(k(B,1^{(n-2)},D),1^{(n-2)},M)$

$\hspace{2.4cm}=g(B,1^{(n-2)},g(D,1^{(n-2)},M))$

$\hspace{2.4cm}=g(B,1^{(n-2)},g(R,1^{(n-2)},x))$

$\hspace{2.4cm}=g(k(B,R,1^{(n-2)}),1^{(n-2)},x)$

$\hspace{2.4cm}=g(B,1^{(n-2)},x)$\\
(2) Since $M=g(B,1^{(n-2)},M)$, then

$\hspace{2cm}M=g(B,1^{(n-2)},M)$

$\hspace{2.5cm}=g(B,1^{(n-2)},g(B,1^{(n-2)},M))$

$\hspace{2.5cm}=g(k(B^{(2)},1^{(n-2)}),1^{(n-2)},M)$

$\hspace{2.5cm}\subseteq g(B.B,1^{(n-2)},M)$\\
Since $g(B.B,1^{(n-2)},M) \subseteq M$ then $M=g(B.B,1^{(n-2)},M)$
which implies $B.B \in \omega$ then $B \subseteq B.B$. Since $B.B  \subseteq B$ then we get $B=B.B$. \\
(3) If $M=g(P,1^{(n-2)},M)$ for some $n$-ary prime hyperideal $P$ of $R$ then $P \in \omega$ and so $B \subseteq P$. If $M \neq g(P,1^{(n-2)},M)$ then there exists some $m \in M$ such that $m \notin g(P,1^{(n-2)},M)$. Since $M$ is a  multiplication  $(m,n)$-hypermodule,  then we have $g(R,1^{(n-2)},m)=g(C, 1^{(n-2)},M)$ for some hyperideal $C$ of $R$. If $C \subseteq P$ then $g(C,1^{(n-2)},M) \subseteq g(P,1^{(n-2)},M)$ which implies $m \in g(P,1^{(n-2)},M)$ which is a contradiction. Therefore $C \nsubseteq P$. Since for each $m \in M$, $m \in g(B,1^{(n-2)},m)$. Then there exists $b \in B$ such that $m \in f(g(b,1^{(n-2)},m),0^{(n-1)})$. Thus

$ \hspace{1cm}0 \in f(-g(b,1^{(n-2)},m),m,0^{(m-2)})$

 $\hspace{1.3cm} \subseteq f(g(-b,1^{(n-2)},m),g(1^{(n-1)},m),g(0,1^{(n-2)},m))$

$\hspace{1.3cm}=g(h(-b,1,0^{(m-2)}),1^{(n-2)},m)$

$\hspace{1.3cm}\subseteq g(h(-b,1,0^{(m-2)}),1^{(n-2)},g(1^{(n-1)},m))$

$\hspace{1.3cm}\subseteq g(h(-b,1,0^{(m-2)}),1^{(n-2)},g(R,1^{(n-2)},m))$

$\hspace{1.3cm}= g(h(-b,1,0^{(m-2)}),1^{(n-2)},g(C,1^{(n-2)},M))$

$\hspace{1.3cm}= g(k(h(-b,1,0^{(m-2)}),C,1^{(n-2)}),1^{(n-2)},M)$\\
Since $M$ is faithful then
$ k(h(-b,1,0^{(m-2)}),C,1^{(n-2)})=0$ then Since $P$ is an $n$-ary hyperideal of $R$ and $ k(h(-b,1,0^{(m-2)}),C,1^{(n-2)})=\{0\} \subseteq P$ with $C \nsubseteq P$ then we conclude that $h(-b,1,0^{(m-2)}) \subseteq P$. Thus $h(b,h(-b,1,0^{(m-2)}),0^{(n-2)}) \subseteq f(B,P,0^{(n-2)})$ which means $h(b,-b,1,0^{(n-3)}) \subseteq h(B,P,0^{(n-2)})$. \\Since $0 \in h(b,-b,0^{(n-2)})$ then

  $\hspace{0.7cm}1=h(0^{(n-1)},1)$
  
  $\hspace{1cm}=h(0,1,0^{(n-2)})$
  
  $ \hspace{1cm} \subseteq h(h(b,-b,1,0^{(n-3)}),1,0^{(n-2)})$
  
  $ \hspace{1cm} =h(b,-b,1,0^{(n-3)})$

  $  \hspace{1cm} \subseteq h(B,P,0^{(n-2)}) $.\\
  Consequently, $R= h(B,P,0^{(n-2)}) $.\\
  (4) Let $N$ be an subhypermodule of $(m,n)$-hypermodule $M$ over $B$. Since $N=g(B,1^{(n-2)},N)$, by (1), then $N$ is an subhypermodule of $(m,n)$-hypermodule $M$ over $R$. Thus $N=g(D,1^{(n-2)},M)$ for some hyperideal $D$ of $R$. It follows that 
  
  $\hspace{0.9cm}N=g(D,1^{(n-2)},M)$
  
  $\hspace{1.3cm}=g(B,D,1^{(n-2)},M)$
  
  $\hspace{1.3cm}=g(k(B,D,1^{(n-2)}),1^{(n-2)},M)$

$\hspace{1.3cm}\subseteq g(B.D,1^{(n-2)},M)$\\
Since $g(B.D,1^{(n-2)},M) \subseteq N$ then $g(B.D,1^{(n-2)},M) = N$. Since $B.D$ is a hyperideal of $B$, then $M$ is a multiplication $(m,n)$-hypermodule over $B$.\\
(5) Let $g(C,1^{(n-2)},M) = M$, for some proper hyperideal $C$ of $R$. Therefore we get 

$\hspace{1cm} M=g(C,1^{(n-2)},M)$

$\hspace{1.5cm}=g(C,1^{(n-2)},g(B,1^{(n-2)},M))$

$\hspace{1.5cm}=g(k(C,B,1^{(n-2)}),1^{(n-2)},M)$

$\hspace{1.5cm} \subseteq g(C.B,1^{(n-2)},M)$.\\
It follows that $C.B \in \omega$. Since $b=\omega(M)$ then $B \subseteq C.B \subseteq C$.  since $C \subseteq B$ then we get $C=B$. 
\end{proof}
\begin{theorem} \label{18}
Let $(M,f,g)$ be a nonzero $(m,n)$- hypermodule over $(R,h,k)$. Let $A_1=\{ P \in Max(R) \ \vert \ M \neq g(P,1^{(n-2)},M)\}$ and $A_2=\{ P \in Max(R) \ \vert \ S_0 \subseteq P \}$ such that $Max(R)$ is the family of all maximal subhypermodules of $R$. Then
\[J_{(m,n)}(M)=g(\bigcap\{P \ \vert \ P \in A_1\},1^{(n-2)},M)=g(\bigcap\{P \ \vert \ P \in A_2\},1^{(n-2)},M)\]
\end{theorem}
\begin{proof}
 Suppose that $B_1=\bigcap\{P \ \vert \ P \in A_1\}$ and $B_2=\bigcap\{P \ \vert \ P \in A_2\}$. Let $K$ is a maximal subhypermodule of $M$. Theorem 3.16 in \cite{sorc2} shows that  $K=g(P,1^{(n-2)},M) \neq M$, for some maximal hyperideal $P$ of $R$. By Theorem 3.11 (i) in  \cite{sorc2}, 
 
 $\hspace{0.5cm} g(\bigcap\{P \ \vert \ P \in A_2\},1^{(n-2)},M)=\bigcap \{g(P,1^{(n-2)},M) \ \vert \ P \in A_2\}$

 $\hspace{5cm} \subseteq \bigcap \{g(P,1^{(n-2)},M) \ \vert \ P \in A_1\}$

 $\hspace{5cm}  =g(\bigcap\{P \ \vert \ P \in A_1\},1^{(n-2)},M)$

 $\hspace{5cm} =J_{(m,n)}(M)$\\
 Now, let $Q \in A_2$. If $g(Q,1^{(n-2)},M)=M$ then we get $J_{(m,n)}(M) \subseteq g(Q,1^{(n-2)},M)$ which means $J_{(m,n)}(M) \subseteq g(\bigcap\{P \ \vert \ P \in A_2\},1^{(n-2)},M)$. If $g(Q,1^{(n-2)},M) \neq M$ then we conclude that $Q \in A_1$ which implies $J_{(m,n)}(M) \subseteq g(Q,1^{(n-2)},M)$ which means $J_{(m,n)}(M) \subseteq g(\bigcap\{P \ \vert \ P \in A_2\},1^{(n-2)},M)$. 
\end{proof}
\begin{definition} \cite{18}
Let $M$ is an  $(m,n)$-hypermodule over $(R,h,k)$. Let $M_1^t$
be subhypermodules of $M$
and $t = l(m-1)+1$, then $M$ is called a (internal) direct sum $M_1\oplus...\oplus M_t$ if satisfies the following
axioms:\\
(1) $M = f_{(l)}(M_1,...,M_t)$ and\\
(2) $M_i \cap f_{(l)}(M_1,...,M_{i-1},0,M_{i+1},...M_t) = \{0\}$
\end{definition}
\begin{theorem} \label{19}
Let $M=\bigoplus_{\lambda \in \Lambda}M_{\lambda}$ for some $(m,n)$-hypermodules $M_{\lambda}$ over $(R,h,k)$ and $\lambda=l(m-1)+1$. Then $M$ is a multiplication $(m,n)$-hypermodule over $(R,h,k)$ if and only if \\
$(i)$ $M_{\lambda}$ is a multiplication $(m,n)$-hypermodule for each $\lambda \in \Lambda$\\
$(ii)$ for each $\lambda \in \Lambda$ there exists hyperideal $I_{\lambda}$  of $R$ such that $g(I_\lambda,1^{(n-2)},M_\lambda)=M_\lambda$
and $g(I_\lambda,1^{(n-2)},\hat{M_\lambda})=0$ where $\hat{M_\lambda}=\bigoplus_{\mu \neq \lambda} M_{\mu}$.
\end{theorem}
\begin{proof}
$\Longrightarrow$Let $M$ is a multiplication $(m,n)$-hypermodule over $(R,h,k)$ and $\lambda \in \Lambda$ such that $\lambda=l(m-1)+1$. Then there exists hyperideal $I_{\lambda}$ of $R$ such that $g(I_{\lambda},1^{(n-2)},M)=M_{\lambda}$ which implies $g(I_{\lambda},1^{(n-2)},\hat{M_{\lambda}}) \cong g(I_{\lambda},1^{(n-2)}, M/M_{\lambda})=M_{\lambda}/M_{\lambda}=0$. Therefore we have $M_{\lambda}=g(I_{\lambda},1^{(n-2)},M)=g(I_{\lambda},1^{(n-2)},M_{\lambda} \bigoplus \hat{M_{\lambda}})=g(I_{\lambda},1^{(n-2)},M_{\lambda})$.\\ $\Longleftarrow$ Let $P$ be a maximal hyperideal of $R$ such that $\mathcal{X}_P(M_{\lambda})=M_{\lambda}$ for all $\lambda \in \Lambda$. Then we get $M=\bigoplus_{\lambda \in \Lambda}M_{\lambda}=\bigoplus_{\lambda \in \Lambda} \mathcal{X}_P(M_{\lambda})=\mathcal{X}_P (\bigoplus_{\lambda \in \Lambda}M_{\lambda})=\mathcal{X}_P(M).$ If $\mathcal{X}_P(M_{\gamma}) \neq M_{\gamma}$, for some $\gamma \in \Lambda$ then we conclude that $M_\gamma$ is $n$-ary $P-$cyclic, by Theorem 3.8 in \cite{sorc2}. It means there exist $m \in M_{\gamma}$ and $p \in P$ with

 $g(h(1,-p,0^{(m-2)}),1^{(n-2)},M_{\gamma}) \subseteq g(R,1^{(n-2)},m)$. \\Thus we have
 
  $\hspace{1cm} g(I_{\gamma},1^{(n-2)},M)=g(I_{\gamma},1^{(n-2)},M_{\gamma} \bigoplus \hat {M_{\gamma}})$
  
  $\hspace{3.5cm}=g(I_{\gamma},1^{(n-2)},M_{\gamma}) \bigoplus g((I_{\gamma},1^{(n-2)},\hat{M_{\gamma}})$
  
  $\hspace{3.5cm}=g(I_{\gamma},1^{(n-2)},M_{\gamma})$

$\hspace{3.5cm}=M_{\gamma}$. \\It follows that 

$g(h(1,-p,0^{(m-2)}),I_{\gamma},1^{(n-3)},M) =g(h(1,-p,0^{(m-2)}),1^{(n-2)},M_{\gamma})$

 $\hspace{5.1cm}\subseteq g(R,1^{(n-2)},m).$\\
Now let $I_{\gamma} \subseteq P$. It follows that $M_{\gamma}=g(I_{\gamma},1^{(n-2)},M_{\gamma}) \subseteq g(P,1^{(n-2)},M_{\gamma})$ which means $g(P,1^{(n-2)},M_{\gamma})=M_{\gamma}$. By Theorem 3.8 in \cite{sorc2}, we have $\mathcal{X}_(M_{\gamma})=M_{\gamma}$, since $M_{\gamma}$ is a multiplication $(m,n)$-hypermodule. This is a contradiction. Thus $I_{\gamma} \nsubseteq P$ and so $g(h(1,-p,0^{(m-2)}),1^{(n-2)},I_{\gamma}) \nsubseteq P$.Hence  $h(1,-q,0^{(m-2)}) \subseteq k(h(1,-p,0^{(m-2)}),I_{\gamma},1^{(n-2)})$
for some $q \in P$. Then 

$g(h(1,-q,0^{(m-2)}),1^{(n-2)},M) \subseteq g(k(h(1,-p,0^{(m-2)}),I_{\gamma},1^{(n-2)}),1^{(n-2)},M)$

$\hspace{4.6cm}=g(h(1,-p,0^{(m-2)}),I_{\gamma},1^{(n-3)},M)$ 

$\hspace{4.6cm}\subseteq g(R,1^{(n-2)},m)$. \\It means $M$ is $n$-ary $P$-cyclic. Consequently, $M$ is a multiplication $(m,n)$-hypermodule, by Theorem 3.8
in \cite{sorc2}. 
\end{proof}
\section{The primary subhypermodules
of multiplication $(m,n)$-hyperodules }

\begin{lem} \label{80}
Let $(M,f,g)$ be a faithful multiplication $(m,n)$-hypermodule over $(R,h,k)$ such that $Q$ is an $n$-ary primary hyperideal of $R$. If $g(r,1^{(n-2)},m) \subseteq g(Q,1^{((n-2)},M)$, for some $r \in R$ , $m \in M$ then $r \in {\sqrt{Q}}^{(m,n)}$ or $m \in g(Q,1^{((n-2)},M)$. 
\end{lem}
\begin{proof}
Let $r \notin {\sqrt{Q}}^{(m,n)}$. Suppose that $E=\{a \in R \ \vert \ g(a,1^{(n-2)},m) \subseteq g(Q,1^{((n-2)},M)\}$. Let $E \neq R$. Therefore there exists a maximal hyperideal $P$ of $R$ such that $E \subseteq P$. If $m \in \mathcal{X}_P(M)$, then for some $p \in P$, $g(h(1,-p,0^{(m-2)}),1,m)=\{0\}$. This means that $g(h(1,-p,0^{(m-2)}),1,m) \subseteq g(Q,1^{((n-2)},M)$. Hence  $h(1,-p,0^{(m-2)}) \subseteq E \subseteq Q$ and so $1 \in Q$, which is a contradiction. Thus $m \notin \mathcal{X}_P(M)$. Then $M$ is n-ary $P$-cyclic, by Theorem 3.8 in \cite{sorc2}. Then there exists $x \in M$, $p \in P$ such
that 

$g(h(1,-p,0^{(m-2)}),1^{(n-2)},M) \subseteq g(R,1^{(n-2)},x)$.\\ This means 
$g(h(1,-p,0^{(m-2)}),1^{(n-2)},m) = g(U,1^{(n-2)},x)$ for some $U \subseteq R$. We have

$g(r,h(1,-p,0^{(m-2)}),1^{(n-3)},m)=g(h(1,-p,0^{(m-2)}),1^{(n-2)},g(r,1^{(n-2)},m))$

$\hspace*{4.8cm}\subseteq g(h(1,-p,0^{(m-2)}),1^{(n-2)},g(Q,1^{(n-2)},M))$

$\hspace*{4.8cm}\subseteq g(Q,1^{(n-2)},g(h(1,-p,0^{(m-2)}),1^{(n-2)},M))$

$\hspace*{4.8cm}\subseteq g(Q,1^{(n-2)},g(R,1^{(n-2)},x))$

$\hspace*{4.8cm}=g(Q,1^{(n-2)},x)$\\
Thus for some $Q_1 \subseteq Q$, we get

$g(r,h(1,-p,0^{(m-2)}),1^{(n-3)},m)=g(Q_1,1^{(n-2)},x)$\\
Also,

$g(k(r,U,1^{(n-2)}),1^{(n-2)},x)=g(r,1^{(n-2)},g(U,1^{(n-2)},x))$

$\hspace*{4.1cm}=g(r,1^{(n-2)},g(h(1,-p,0^{(m-2)}),1^{(n-2)},m))$

$\hspace*{4.1cm}=g(r,1^{(n-3)},h(1,-p,0^{(m-2)}),m)$

$\hspace*{4.1cm}\subseteq g(Q,1^{(n-2)},x).$\\
Let $u \in U$. Then for some  $Q_1 \subseteq Q$ we have 

$g(k(r,u,1^{(n-2)}),1^{(n-2)},x)=g(Q_1,1^{(n-2)},x)$

$\hspace*{4.1cm}=f(g(Q_1,1^{(n-2)},x),0^{(m-1)})$\\
This means $0 \in f(g(k(r,u,1^{(n-2)}),1^{(n-2)},x),-g(Q_1,1^{(n-2)},x),0^{(m-2)}))$

$\hspace*{1.7cm}=g(h(k(r,u,1^{(n-2)}),-Q_1,0^{(m-2)}),1^{(m-2)},x)$.\\
Therefore $0 \in h(k(r,u,1^{(n-2)}),-Q_1,0^{(m-2)})$, since $M$ is faithful. This implies that $k(r,u,1^{(n-2)}) \in Q_1$. Thus

$k(u,k(h(1,-p,0^{(m-2)}),r,1^{(n-2)}),1^{(n-2)})=k(k(u,r,1^{(n-2)}),h(1,-p,0^{(m-2)}),1^{(n-2)})$

$\hspace*{6.2cm}\subseteq k(Q_1,h(1,-p,0^{(m-2)}),1^{(n-2)})$

$\hspace*{6.2cm}\subseteq k(Q,h(1,-p,0^{(m-2)}),1^{(n-2)})$

$\hspace*{6.2cm}\subseteq Q$.\\
Since $Q$ is an n-ary primary hyperideal of $R$, we have 

$u \in Q$ or $k(1,k(h(1,-p,0^{(m-2)}),r,1^{(n-2)}),1^{(n-2)}) \in {\sqrt{Q}}^{(m,n)}$.\\
Let $k(1,k(h(1,-p,0^{(m-2)}),r,1^{(n-2)}),1^{(n-2)}) =k(h(1,-p,0^{(m-2)}),r,1^{(n-2)})\in {\sqrt{Q}}^{(m,n)}$. Since $Q$ is an n-ary primary hyperideal of $R$, then ${\sqrt{Q}}^{(m,n)}$ is an n-ary prime hyperideal, by Theorem 4.28 in \cite{sorc1}. Since $r \notin {\sqrt{Q}}^{(m,n)}$ and $1 \notin {\sqrt{Q}}^{(m,n)}$, then $h(1,-p,0^{(m-2)}) \in {\sqrt{Q}}^{(m,n)} \subseteq Q \subseteq E \subseteq P$ and so $1 \in P$, a contradiction. Then we conclude that $u \in Q$. Therefore $g(h(1,-p,0^{(m-2)}),1^{(n-2)},m) = g(U,1^{(n-2)},m) \subseteq g(Q,1^{(n-2)},M)$. By the definition of the set $E$, we have $h(1,-p,0^{(m-2)}) \subseteq E \subseteq P$. This is a contradiction. Consequently, $E=R$ and so $m \in g(1^{(n-1)},m) \subseteq g(Q,1^{(n-2)},M)$.
\end{proof}
\begin{corollary}
Let $(M,f,g)$ be a faithful multiplication $(m,n)$-hypermodule over $(R,h,k)$. Let $N$ be a subhypermodule of $M$ such that $N=g(Q,1^{(n-2)},M)$and $M \neq g(Q,1^{(n-2)},M)$ for some n-ary primary hyperideal $Q$ of $R$. Then $N$ is an $n$-ary primary subhypermodule of $M$. 
\end{corollary}
\begin{proof}
Suppose that $N$ be a subhypermodule of a faithful multiplication $(m,n)$-hypermodule $M$ such that $N=g(Q,1^{(n-2)},M)$ for some n-ary primary hyperideal $Q$ of $R$. Let for $r_1^{n-1} \in R$ and $m \in M \backslash N$, $g(r_1^{n-1},m) \subseteq N$. Since $N=g(Q,1^{(n-2)},M)$ and $g(k(r_1^{n-1},1),1^{(n-2)},m)=g(r_1^{n-1},m) \subseteq g(Q,1^{(n-2)},M)$, we have  $k(r_1^{n-1},1) \in \sqrt{Q}^{(m,n)}$, By Lemma \ref{80}. Since $\sqrt{Q}^{(m,n)} \subseteq \sqrt{S_N}^{(m,n)}$ with $S_N=\{r \in R \ \vert \ g(r,1^{(n-2)},M) \subseteq N\}$, then $N$  is an $n$-ary primary subhypermodule of $M$.  
\end{proof}
\begin{corollary}
Let $(M,f,g)$  be a multiplication  $(m,n)$-hypermodule over $(R,h,k)$. The followings are equivalent for a proper subhypermodule
$N$ of $M$:\\
 (1) $N$ is an $n$-ary primary subhypermodule of $M$.\\
 (2) $S_N=\{r \in R \ \vert \ g(r,1^{(n-2)},M) \subseteq N\}$ is an $n$-ary primary hyperideal of $R$. \\
 (3)  $N = g(Q,1^{(n-2)},M)$ for some $n$-ary primary hyperideal $Q$ of $R$ with $S_0 \subseteq Q$.
\end{corollary}
\begin{proof}
$(1) \Longrightarrow (2)$ It follows  from Theorem 5.3 in \cite{sorc3}.\\
$(2) \Longrightarrow (3)$ It is clear by Remark 3.2 in \cite{sorc2}.\\
$(3) \Longrightarrow (1)$ Let for $r_1^{n-1} \in R$ and $m \in M \backslash N$, $g(r_1^{n-1},m) \subseteq N$. Then,

$g(1^{(n-1)},g(r_1,1^{(n-2)},g(k(r_2^{n-1},1^{(2)}),1^{(n-2)},m)))$

$\hspace{0.5cm}=g(1^{(n-2)},k(r_1,1^{(n-1)},g(k(r_2^{n-1},1^{(2)}),1^{(n-2)},m))$

$\hspace{0.5cm} \vdots$

$\hspace{0.5cm}=g(1,k(r_1,1^{(n-1)}),...,k(r_{n-2},1^{(n-1)}),g(r_{n-1},1^{(n-2)},m))$

$\hspace{0.5cm}=g(k(r_1,1^{(n-1)}),...,k(r_{n-2},1^{(n-1)}),k(r_{n-1},1^{(n-1)}),m)$

$\hspace{0.5cm}=(r_1^{n-1},m)$

 $\hspace{0.5cm}\subseteq N=g(Q,1^{(n-2)},M)$\\
 for n-ary primary hyperideal $Q$ of $R$. Since $1 \notin \sqrt{Q}^{(m,n)}$, then 
 
 $g(r_1,1^{(n-2)},g(k(r_2^{n-1},1^{(2)}),1^{(n-2)},m)) \subseteq g(Q,1^{(n-2)},M)$\\
 by Lemma \ref{80}. By using Lemma \ref{80} again, we have 
 
 $r_1 \in \sqrt{Q}^{(m,n)}$ or $g(k(r_2^{n-1},1^{(2)}),1^{(n-2)},m) \subseteq g(Q,1^{(n-2)},M)$\\
In the second case, we have 
 
 $g(r_2, g(r_3^{n-1},1^{(2)}),1^{(n-3)},m)) \subseteq g(Q,1^{(n-2)},M)$\\
 Therefore we get

 $r_2 \in \sqrt{Q}^{(m,n)}$ or $g(r_3^{n-1},1^{(2)}),1^{(n-3)},m) \subseteq g(Q,1^{(n-2)},M)$,\\
 by Lemma \ref{80}.  In the second case, by using a similar argument we get $r_{n-1} \in  \sqrt{Q}^{(m,n)}$ or $m \in g(Q,1^{(n-1)},M)=N$.  Since $m \notin  N$, then we have $r_i \in  \sqrt{Q}^{(m,n)}$ for some $1 \leq i \leq n-1$ and so  $k(r_1^{n-1},1) \in \sqrt{Q}^{(m,n)}$.  Since $\sqrt{Q}^{(m,n)} \subseteq \sqrt{S_N}^{(m,n)}$, then $k(r_1^{n-1},1) \in \sqrt{S_N}^{(m,n)}$. It follows that 
 $N$ is an $n$-ary primary subhypermodule of $M$.
\end{proof}

\section{ intersections and  sums of the multiplication $(m,n)$-hypermodule}
\begin{theorem} \label{21}
Let $(M,f,g)$  be an $(m,n)$-hypermodule over $(R,h,k)$ such that $M=f_{(l)}(N_{\lambda}^{(\lambda \in \Lambda)})$ for some multiplication subhypermodules $N_{\lambda}$ of $M$ and $\lambda =l(m-1)+1$. The followings are equivalent:\\
$(1)$ $M$ is a multiplication $(m,n)$-hypermodule.\\
$(2)$ $N_{\lambda}=g(S_{N_{\lambda}},1^{(n-2)},M)$, for each $\lambda \in \Lambda$\\
$(3)$ $R=h(A(m),B,0^{(m-2)})$  for each $m \in M$ such that $A(m)=\{r \in R \ \vert \ 0 \in g(r,1^{(n-2)},m)\}$ and $B=f_{(l)}(S_{N_{\lambda}}^{(\lambda \in \Lambda)})$\\
$(4)$ for every maximal hyperideal $P$ of $R$ either $M=\mathcal{X}_P(M)$ or there exist $ a \in \bigcup_{\lambda \in \Lambda} N_{\lambda}$ and $p \in P$ such that $g(h(1,-p,0^{(m-2)}),1^{(n-2)},M) \subseteq g(R,1^{(n-2)},a)$
\end{theorem}
\begin{proof}
$(1) \Longrightarrow (2)$ It is clear.\\
$(2) \Longrightarrow (3)$ Let $R \neq h(A(m),B,0^{(m-2)})$ for some $m \in M$. Then  there exists some maximal hyperideal $Q$ of $R$ such that $h(A(m),B,0^{(m-2)}) \subseteq Q$. Then $B \subseteq Q$ which implies $g(B,1^{(n-2)},M) \subseteq g(Q,1^{(n-2)},M)$ which means $g(f_{(l)}(S_{N_{\lambda}}^{(\lambda \in \Lambda)}),1^{(n-2)},M) \subseteq g(Q,1^{(n-2)},M)$. Thus $f_{(l)}(N_{\lambda}^{(\lambda \in \Lambda)})=M \subseteq g(Q,1^{(n-2)},M)$ and so $M=g(Q,1^{(n-2)},M)$. Since $m \in M=f_{(l)}(N_{\lambda}^{(\lambda \in \Lambda)})$ then $m \in f_{(l)}(m_{\lambda}^{(\lambda \in \Lambda)})$ for some elements $m_{\lambda} \in N_{\lambda}$. Since $N_{\lambda}$ is a multiplication $(m,n)$-hypermodule then there exists hyperideal $I_{\lambda}$ of $R$ such that $g(R, 1^{(n-2)},m_{\lambda})=g(I_{\lambda},1^{(n-2)},N_{\lambda})$.
Thus we get

$\hspace{1cm}g(R, 1^{(n-2)},m_{\lambda})=g(I_{\lambda},1^{(n-2)},N_{\lambda})$.

$\hspace{3.6cm}=g(I_{\lambda},1^{(n-2)},g(S_{N_{\lambda}},1^{(n-2)},M)$.

$\hspace{3.6cm}=g(I_ {\lambda},S_{N_{\lambda}},1^{(n-3)},M)$

$\hspace{3.6cm}=g(I_ {\lambda},S_{N_{\lambda}},1^{(n-3)},g(Q,1^{(n-2)},M))$

$\hspace{3.6cm}=g(I_ {\lambda},Q,1^{(n-3)},g(S_{N_{\lambda}},1^{(n-2)},M))$

$\hspace{3.6cm} \subseteq g(I_ {\lambda},Q,1^{(n-3)},{N_{\lambda}})$

$\hspace{3.6cm}=g(Q,1^{(n-2)},g(I_ {\lambda},1^{(n-2)},{N_{\lambda}}))$

$\hspace{3.6cm}=g(Q,1^{(n-2)},g(R, 1^{(n-2)},m_{\lambda}))$

$\hspace{3.6cm}=g(Q,R, 1^{(n-3)},m_{\lambda})$

$\hspace{3.6cm}=g(Q, 1^{(n-2)},m_{\lambda})$. \\
Then there exists $q_{\lambda} \in Q$ such that $m_{\lambda} \in f(g(q_{\lambda},1^{(n-2)},m_{\lambda}),0^{(m-1)})$ which implies 

$ \hspace{1cm} 0 \in f(-g(q_{\lambda},1^{(n-2)},m_{\lambda}),m_{\lambda},0^{(m-2)})$

$\hspace{1.3cm}\subseteq f(g(-q_{\lambda},1^{(n-2)},m_{\lambda}),g(1^{(n-1)},m_{\lambda})
,g(0,1^{(n-2)},m_{\lambda}))$

$\hspace{1.3cm}=g(h(-q_{\lambda},1,0^{(m-2)}),1^{(n-2)},m_{\lambda}))$.
\\Thus we get $0 \in g(g(h(-q_{\lambda},1,0^{(m-2)})^{(\lambda \in \Lambda)}),1^{(n-2)},m)$ which means 

$g(h(-q_{\lambda},1,0^{(m-2)})^{(\lambda \in \Lambda)}) \in A(m) \subseteq Q$. \\Since all $q_{\lambda}$'s belong to $Q$ then we conclude that $1 \in Q$ which is a contradiction. Consequently,  
$R= h(A(m),B,0^{(m-2)})$.\\
$(3)\Longrightarrow (4)$ Suppose that for some maximal hyperideal $P$ of $R$,  $M \neq \mathcal{X}_P(M)$. Then there exists some $m \in M$ such that $0 \notin g(h(1,-p,0^{(m-2)}),1^{(n-2)},m)$ for all $p \in P$. If $A(m) \nsubseteq  P$ then there exists $p \in P$ such that $h(1,-p,0^{(m-2)}) \subseteq  A(m)$ which means $0 \in g(h(1,-p,0^{(m-2)}),1^{(n-2)},m)$ which is a contradiction. Then $A(m) \subseteq P$ and so $B \notin P$ by (3). This follows that $S_{N_{\lambda}} \nsubseteq P$ for some $\lambda \in \Lambda$. Then there exists  $p \in P$ such that $h(1,-p,0^{(m-2)}) \subseteq  S_{N_{\lambda}}$. Thus $g(h(1,-p,0^{(m-2)}),1^{(n-2)},M) \subseteq g(S_{N_{\lambda}},1^{(n-2)},M) \subseteq N_{\lambda}$. Let $\mathcal{X}_P(N_{\lambda})=N_{\lambda}$ and  $x \in g(S_{N_{\lambda}},1^{(n-2)},m) \subseteq N_{\lambda}$. Then for all $p \in P$ we have \\ $0 \in g(h(1,-p,0^{(m-2)}),1^{(n-2)},x) \subseteq g(h(1,-p,0^{(m-2)}),1^{(n-2)},g(S_{N_{\lambda}},1^{(n-2)},m))$

$\hspace{4.7cm}=g(g(h(1,-p,0^{(m-2)}),S_{N_{\lambda}},1^{(n-2)}),1^{(n-3)},m)$\\
Therefore $g(h(1,-p,0^{(m-2)}),S_{N_{\lambda}},1^{(n-2)}) \subseteq A(m) \subseteq P$. Since $P$ is an n-ary prime hyperideal and $h(1,-p,0^{(m-2)}) \nsubseteq P$ then we get $S_{N_{\lambda}} \subseteq P$ which is a contradiction. Then $\mathcal{X}_P(N_{\lambda}) \neq N_{\lambda}$. Hence $N_{\lambda}$ is  n-ary $P$-cyclic, by Theorem 3.8 in \cite{sorc2}. Then there exist $a \in  N_{\lambda}$ and $q \in Q$ such that $g(h(1,-p,0^{(m-2)}),1^{(n-2)},N_{\lambda}) \subseteq g(R,1^{(n-2)},a)$. By $g(h(1,-p,0^{(m-2)}),1^{(n-2)},M) \subseteq g(S_{N_{\lambda}},1^{(n-2)},M) \subseteq N_{\lambda}$ it follows that $g(h(1,-q,0^{(m-2)}),h(1,-p,0^{(m-2)}),1^{(n-3)},M) \subseteq g(R,1^{(n-2)},a)$. It is clear that $h(p,q,-k(p,q,1^{(n-2)}),0^{(m-3)}) \subseteq P$. \\Let $t \in h(p,q,-k(p,q,1^{(n-2)}),0^{(m-3)})$. Then 

$ g(h(1,-t,0^{(m-2)}),1^{(n-2)},M)$

$ \hspace{0.7cm} \subseteq g(h(1,-(h(p,q,-k(p,q,1^{(n-2)}),0^{(m-3)}),0^{(m-2)}),1^{(n-2)},M)$

$\hspace{0.7cm}=g(h(1,-p,-q,k(p,q,1^{(n-2)}),0^{(m-4)}),1^{(n-2)},M)$

$\hspace{0.7cm}=g(h(1,-p,h(-q,k(p,q,1^{(n-2)}),0^{(m-2)})),0^{(m-3)}),1^{(n-2)},M)$

$\hspace{0.7cm}=g(h(1,-p,h(-q,-(-k(p,q,1^{(n-2)})),0^{(m-2)})),0^{(m-3)}),1^{(n-2)},M)$

$\hspace{0.7cm}=g(h(1,-p,h(-q,k(-p,-q,1^{(n-2)})),0^{(m-2)})),0^{(m-3)}),1^{(n-2)},M)$

$\hspace{0.7cm}=g(h(1,-p,k(-q,h(1,-p,0^{(m-2)}),1^{(n-2)})),0^{(m-3)}),1^{(n-2)},M)$

$\hspace{0.7cm}=g(h(h(1,-p,0^{(m-2)}),k(-q,h(1,-p,0^{(m-2)}),1^{(n-2)})),0^{(m-2)}),1^{(n-2)},M)$

$\hspace{0.7cm}=g(k(h(1,-p,0^{(m-2)}),h(1,-q,0^{(m-2)}),1^{(m-2)}),1^{(n-2)},M)$

$\hspace{0.7cm} \subseteq g(R,1^{(n-2)},a)$.\\

$(4) \Longrightarrow (1) $ This is obvious by Theorem 3.8 in \cite{sorc2}.
\end{proof}
\begin{corollary} \label{22}
Let $(M,f,g)$  be an $(m,n)$-hypermodule over $(R,h,k)$ such that $R=h_{(l)}(S_{N_{\lambda}}^{(\lambda \in \Lambda)})$ for some multiplication subhypermodules $N_{\lambda}$ of $M$ and $\lambda =l(m-1)+1$. Then $M$ is a multiplication hypermodule.
\end{corollary}
\begin{proof}
Since

$\hspace{1cm} M=g(R,1^{(n-2)},M)$

 $\hspace{1.5cm} =g(h_{(l)}(S_{N_{\lambda}}^{(\lambda \in \Lambda)}),1^{(n-2)},M)$

$\hspace{1.5cm} =f_{(l)}(g(S_{N_{\lambda}},1^{(n-2)},M)^{(\lambda \in \Lambda)}))$

$\hspace{1.5cm}  \subseteq f_{(l)}(N_{\lambda}^{(\lambda \in \Lambda)})$

$\hspace{1.5cm}  \subseteq M$,\\
then $M=f_{(l)}(N_{\lambda}^{(\lambda \in \Lambda)})$. Thus $M$ is a multiplication hypermodule, by Theorem \ref{21}.
\end{proof}
\begin{theorem} \label{23}
Let $(M,f,g)$  be an $(m,n)$-hypermodule over $(R,h,k)$ such that $M=f_{(l)}(N_{\lambda}^{(\lambda \in \Lambda)})$ for some multiplication subhypermodules $N_{\lambda}$ of $M$ and $\lambda =l(m-1)+1$. Then for each subhypermodule $N$ of $M$, $N=f_{(l)}((N \cap N_{\lambda})^{(\lambda \in \Lambda)})$.
\end{theorem}
\begin{proof}
Let $N$ be a arbitrary subhypermodule  of $M$. Then 

$\hspace{1cm} N=g(S_N,1^{(n-2)},M)$

$\hspace{1.4cm} =g(S_N,1^{(n-2)},f_{(l)}(N_{\lambda}^{(\lambda \in \Lambda)}))$

$\hspace{1.4cm} =f_{(l)}((g(S_N,1^{(n-2)},N_{\lambda})^{(\lambda \in \Lambda)})$

$\hspace{1.4cm} \subseteq f_{(l)}((N \cap N_{\lambda})^{(\lambda \in \Lambda)})$

$\hspace{1.4cm} \subseteq N$. \\
Consequently, $N=f_{(l)}((N \cap N_{\lambda})^{(\lambda \in \Lambda)})$. 
\end{proof}
\begin{theorem} \label{new3}
Let $(M,f,g)$  be an $(m,n)$-hypermodule over $(R,h,k)$. Let $H$ and $K$ be two multiplication subhypermodules of $M$ such that $H$, $K$ and $f(H,K,1^{(n-2)})$ are multiplication hypermodules .Then $H \cap K$ is  a multiplication hypermodule.
\end{theorem}
\begin{proof}
Let $\mathcal{X}_P(H \cap K) \neq H \cap K$ for some maximal hyperideal $P$ of $R$. Then we conclude that $\mathcal{X}_P(H) \neq H $, $\mathcal{X}_P( K) \neq  K$ and $\mathcal{X}_P(f(H,K,0^{(m-2)})) \neq f(H,K,0^{(m-2)})$. By Theorem 3.8 in \cite{sorc2}  and Theorem \ref{21}, there exist $a \in H$, $b \in K$, $c \in H \cup K$ and $p,p^\prime ,p^{\prime \prime} \in P$ such that

$\hspace{1cm} g(h(1,-p,0^{(m-2)}),1^{(n-2)},H) \subseteq g(R,1^{(n-2)},a)$

$\hspace{1cm} g(h(1,-p^\prime,0^{(m-2)}),1^{(n-2)},K) \subseteq g(R,1^{(n-2)},b)$\\
and 

$\hspace{1cm} g(h(1,-p^{\prime \prime},0^{(m-2)}),1^{(n-2)},f(H,K,0^{(m-2)})) \subseteq g(R,1^{(n-2)},c)$.\\
If $c \in H$ then \\

$\hspace{0.3cm} g(h(1,-p^{\prime \prime},0^{(m-2)}),1^{(n-2)},b) \subseteq g(h(1,-p^{\prime \prime},0^{(m-2)}),1^{(n-2)},f(H,K,0^{(m-2)}))$

$ \hspace{4.9cm}\subseteq g(R,1^{(n-2)},c)$

$ \hspace{4.9cm} \subseteq H$\\
which implies 

$\hspace{1cm} g(h(1,-p^{\prime \prime},0^{(m-2)}),1^{(n-2)},b) \subseteq H \cap K$.\\
Let $s \in h(p^{\prime},p^{\prime \prime},-k(p^{\prime},p^{\prime \prime},1^{(n-2)}),0^{(m-3)}) \subseteq  P$. Then\\
$ g(h(1,-s,0^{(m-2)}),1^{(n-2)},H \cap K)$

$  \subseteq g(h(1,-(h(p^{\prime},p^{\prime \prime},-k(p^{\prime},p^{\prime \prime},1^{(n-2)}),0^{(m-3)}),0^{(m-2)}),1^{(n-2)},H \cap K)$

$=g(h(1,-p^{\prime},-p^{\prime \prime},k(p^{\prime},p^{\prime \prime},1^{(n-2)}),0^{(m-4)}),1^{(n-2)},H \cap K)$

$=g(h(1,-p^{\prime},h(-p^{\prime \prime},k(p^{\prime},p^{\prime \prime},1^{(n-2)}),0^{(m-2)})),0^{(m-3)}),1^{(n-2)},H \cap K)$

$=g(h(1,-p^{\prime},h(-p^{\prime \prime},-(-k(p^{\prime},p^{\prime \prime},1^{(n-2)})),0^{(m-2)})),0^{(m-3)}),1^{(n-2)},H \cap K)$

$=g(h(1,-p^{\prime},h(-p^{\prime \prime},k(-p^{\prime},-p^{\prime \prime},1^{(n-2)})),0^{(m-2)})),0^{(m-3)}),1^{(n-2)},H \cap K)$

$=g(h(1,-p^{\prime},k(-p^{\prime \prime},h(1,-p^{\prime},0^{(m-2)}),1^{(n-2)})),0^{(m-3)}),1^{(n-2)},H \cap K)$

$=g(h(h(1,-p^{\prime},0^{(m-2)}),k(-p^{\prime \prime},h(1,-p^{\prime},0^{(m-2)}),1^{(n-2)})),0^{(m-2)}),1^{(n-2)},H \cap K)$

$=g(k(h(1,-p^{\prime},0^{(m-2)}),h(1,-p^{\prime \prime},0^{(m-2)}),1^{(n-2)}),1^{(n-2)},H \cap K)$

$ \subseteq g(h(1,-p^{\prime \prime },0^{(m-2)}),1^{(n-2)},g(h(1,-p^{\prime },0^{(m-2)}),1^{(n-2)},  K))$

$ \subseteq g(h(1,-p^{\prime \prime },0^{(m-2)}),1^{(n-2)},g(R,1^{(n-2)},b))$

$ = g(R,1^{(n-2)},g(h(1,-p^{\prime \prime },0^{(m-2)}),1^{(n-2)},b))$.\\
Then there exists  $m \in g(h(1,-p^{\prime \prime },0^{(m-2)}),1^{(n-2)},b) \subseteq H \cap K$ such that 

$g(h(1,-s,0^{(m-2)}),1^{(n-2)},H \cap K) \subseteq g(R,1^{(n-2)},m)$\\
which means $H \cap K$ is n-ary $P$-cyclic. Similarly, by $ c \in K$ it follows that $H \cap K$ is n-ary $P$-cyclic. Thus $H \cap K$ is a multiplication hypermodule, by Theorem 3.8 in \cite{sorc2}. 
\end{proof}
\begin{lem} \label{new1}
Let $(M,f,g)$  be an $(m,n)$-hypermodule over $(R,h,k)$. Let  subhypermodules $N_1$ and $N_2$ of $M$ be multiplication. If  $R=h(A,B,0^{(m-2)})$ such that $A=\{r \in R \ \vert \ g(r,1^{(n-2)},N_2) \subseteq N_1\}$ and $B=\{r \in R \ \vert \ g(r,1^{(n-2)},N_1) \subseteq N_2\}$, then $f(N_1,N_2,0^{(m-2)})$ is a multiplication $(m,n)$-hypermodules.
\end{lem}
\begin{proof}
This follows from Corollary \ref{22}.
\end{proof}
\begin{theorem} \label{new2}
Let $(M,f,g)$  be an $(m,n)$-hypermodule over $(R,h,k)$. Let $L,N_1,...,N_t$  be subhypermodules of
$M$ such that $L$, $N_i$, $f(L,N_i,0^{(m-2)})$ $(1 \leq i \leq t)$ and $N=\bigcap_{i=1}^t N_i$    are multiplication $(m,n)$-hypermodules. Then $f(L,N,0^{(m-2)})$ is a multiplication $(m,n)$-hypermodule.
\end{theorem}
\begin{proof}
Let  $\mathcal{X}_P (f(L,N,0^{(m-2)})) \neq   f(L,N,0^{(m-2)})$  for some maximal hyperideal $P$ of $R$. Then $\mathcal{X}_P(f(L,N_i,0^{(m-2)})
\neq f(L,N_i,0^{(m-2)})$ for each $1 \leq i \leq t$, since $f(L,N,0^{(m-2)}) \subseteq f(L,N_i,0^{(m-2)})$. By  Theorem  \ref{21} we conclude that  $h(A_i,B_i,0^{(m-2)}) \nsubseteq P$ with $A_i=\{r \in R \ \vert \ g(r,1^{(n-2)},N_i) \subseteq L\} $ and $B_i=\{r \in R \ \vert \ g(r,1^{(n-2)},L) \subseteq N_i\} $. Let $C=\{r \in R \ \vert \ g(r,1^{(n-2)},N) \subseteq L\} $. Then $A_i \subseteq C$. Therefore we have $h(B_i,C,0^{(m-2)}) \nsubseteq P$ for each $1 \leq i \leq t$. Let $B=\{r \in R \ \vert \ g(r,1^{(n-2)},L) \subseteq N\} $. Then we get 
$h(B,C,0^{(m-2)})=h(\bigcap_{i=1}^t B_i,C,0^{(m-2)}) \nsubseteq P$ which means $B \nsubseteq P$ or $C \nsubseteq P$. Hence there exists $p \in P$ such that $h(1,-p,0^{(m-2)}) \subseteq B \cup C=B^{\prime} \cup C^{\prime}$ with $ B^{\prime}= \{r \in R \ \vert \ g(r,1^{(n-2)},f(L,N,0^{(m-2)})) \subseteq N \}$ and $ C^{\prime}=\{r \in R \ \vert \ g(r,1^{(n-2)},f(L,N,0^{(m-2)})) \subseteq L\}$. Then we have 

$\hspace{0.5cm} g(h(1,-p,0^{(m-2)}),1^{(n-2)},f(L,N,0^{(m-2)})) \subseteq L$\\ or 

$\hspace{0.5cm} g(h(1,-p,0^{(m-2)}),1^{(n-2)},f(L,N,0^{(m-2)})) \subseteq N$.\\ Since $L$ and $N$ are multiplication $(m,n)$-hypermodules, then by Lemma  \ref{new1}, we conclude that $f(L,\bigcap_{i=1}^t N_i,0^{(m-2)})=f(L,N,0^{(m-2)})$ is a multiplication $(m,n)$-hypermodule. 
\end{proof}

\begin{theorem} \label{new4}
Let $(M,f,g)$  be an $(m,n)$-hypermodule over $(R,h,k)$. Let $N_1,...,N_n$ be subhypermodules of
$M$ such that for all $1 \leq i <  j \leq n$, $f(N_i,N_j,0^{(m-2)})$ is a  multiplication $(m,n)$-hypermodule. Then $N_1,...,N_n$ are multiplication $(m,n)$-hypermodules if and only if $N=\bigcap_{i=1}^n N_i$ is a  multiplication $(m,n)$-hypermodule. 
\end{theorem}
\begin{proof}
$\Longrightarrow$ Let $N_1,...,N_n$ are multiplication $(m,n)$-hypermodules. By
induction on $n$, $N_2 \cap...\cap N_n$ is a multiplication $(m,n)$-hypermodule. Now, by using  Theorem \ref{new2} and Theorem \ref{new3}, the proof is completed.

$\Longleftarrow $ Let $N=\bigcap_{i=1}^n N_i$ be a multiplication $(m,n)$-hypermodule. Assume that for some maximal hyperideal $P$ of $R$, $\mathcal{X}_P(N_1) \neq N_1$. This implies that for all $2 \leq i \leq n$, $\mathcal{X}_P (f(N_1,N_i,0^{(m-2)})) \neq f(N_1,N_i,0^{(m-2)})$.  Then there exist $x_i \in N_1 \cup N_i$ and $p_i \in P$ such that 
$g(h(1,-p_i,1^{(m-2)}),1^{(n-2)},f(N_1,N_i,0^{(m-2)})) \subseteq g(R,1^{(n-2)},x_i)$, by Theorem \ref{21}. If $x_i \in N_1$ then $N_1$ is an n-ary $P$-cyclic. If $x_i \in N_i$ $(2 \leq i \leq n )$, then 

$g(h(1,-{p_2},0^{(m-2)}),...,h(1,-p_n,0^{(m-2)}),N_1) \subseteq \bigcap_{i=1}^n N_i$.\\
By Theorem 3.8 in \cite{sorc2} and using an argument similar to that in the proof of Theorem \ref{21} $((3) \Longrightarrow (4))$, we conclude  that $N_1$ is multiplication. Similarly, we can show that for all $2 \leq i \leq n$, $N_i$ is multiplication. 
\end{proof}
\begin{corollary}
Let $(M,f,g)$  be an $(m,n)$-hypermodule over $(R,h,k)$ such that subhypermodules $N_1,...,N_n,N_{n+1},...,N_k$ of $M$ are  multiplication. If for all $1 \leq i <  j \leq k$, $f(N_i,N_j,0^{(m-2)})$  is a multiplication $(m,n)$-hypermodule. Then $f(N, L,0^{(m-2)})$ is a multiplication $(m,n)$-hypermodule with $N=N_1 \cap ... \cap N_n$ and $L=N_{n+1}\cap...\cap N_k$.
\end{corollary}
\begin{proof}
By Theorem \ref{new4}, $N$ is a multiplication $(m,n)$-hypermodule. Theorem \ref{new2} shows that for each $n+1 \leq i \leq k$,  $f(N,N_i,0^{(m-2)})$ is a multiplication $(m,n)$-hypermodule. Thus $f(N, N_{n+1}\cap...\cap N_k,0^{(m-2)})$
is a multiplication $(m,n)$-hypermodule, by Theorem \ref{new2}.
\end{proof}

\end{document}